\documentclass[12pt,reqno]{article}

\usepackage[usenames]{color}
\usepackage{amssymb}
\usepackage{amscd}

\usepackage[colorlinks=true,
linkcolor=webgreen,
filecolor=webbrown,
citecolor=webgreen]{hyperref}

\definecolor{webgreen}{rgb}{0,.5,0}
\definecolor{webbrown}{rgb}{.6,0,0}

\usepackage{color}
\usepackage{fullpage}
\usepackage{float}

\usepackage{graphics,amsmath,amssymb}
\numberwithin{equation}{section}
\usepackage{amsthm}
\usepackage{amsfonts}
\usepackage{latexsym}

\newcommand{\seqnum}[1]{\href{http://oeis.org/#1}{\underline{#1}}}

\begin{document}


\theoremstyle{plain}
\newtheorem{theorem}{Theorem}
\newtheorem{corollary}[theorem]{Corollary}
\newtheorem{lemma}{Lemma}
\theoremstyle{remark}
\newtheorem{remark}[theorem]{Remark}
\theoremstyle{plain}
\newtheorem*{example}{Examples}

\newcommand{\lrf}[1]{\left\lfloor #1\right\rfloor}

\begin{center}
\vskip 1cm{\LARGE\bf Fibonacci-Catalan Series \\
\vskip .11in }

\vskip 1cm

{\large

Kunle Adegoke \\
Department of Physics and Engineering Physics, \\ Obafemi Awolowo University, 220005 Ile-Ife, Nigeria \\
\href{mailto:adegoke00@gmail.com}{\tt adegoke00@gmail.com}

\vskip 0.2 in

Robert Frontczak\footnote{Statements and conclusions made in this article by R. F. are entirely those of the author. They do not necessarily reflect the views of LBBW.} \\
Landesbank Baden-W\"urttemberg (LBBW), Stuttgart,  Germany \\
\href{mailto:robert.frontczak@lbbw.de}{\tt robert.frontczak@lbbw.de}

\vskip 0.2 in

Taras Goy  \\
Faculty of Mathematics and Computer Science\\
Vasyl Stefanyk Precarpathian National University, Ivano-Frankivsk, Ukraine\\
\href{mailto:taras.goy@pnu.edu.ua}{\tt taras.goy@pnu.edu.ua}}
\end{center}

\vskip .2 in

\begin{abstract}
We study certain series with Catalan numbers and reciprocal Catalan numbers, respectively,
and provide seemingly new closed form evaluations of these series with Fibonacci (Lucas) entries.
In addition, we state some combinatorial sums that can be inferred from the series.
\end{abstract}

\section{Introduction and Motivation}

The famous Catalan numbers $C_n, n\geq 0,$ are defined by $C_n = \frac{1}{n+1}\binom {2n}{n}$.
The numbers are indexed as sequence \seqnum{A000108} in the On-Line Encyclopedia of Integer Sequences \cite{OEIS}.
They have the generating function \cite{koshy2,stanley}
\begin{equation*}
G(z) = \sum_{n=0}^\infty C_n z^n = \frac{1-\sqrt{1-4z}}{2z}
\end{equation*}
and possess, among other fascinating properties, the integral representations \cite{qiguo}
\begin{equation*}
C_n = \frac{1}{2\pi} \int_0^4 z^{n}\sqrt{\frac{4-z}{z}} dz \qquad\mbox{and}\qquad C_n = \frac{1}{\pi} \int_0^2 z^{2n}\sqrt{4-z^2} dz.
\end{equation*}
Consult also the articles by Dana-Picard \cite{dana1,dana2,dana3} for these representations.
The reciprocals of Catalan numbers are generated by the function
\begin{equation*}
f(z) = \sum_{n=0}^\infty \frac{z^n}{C_n}, \qquad |z| < 4,
\end{equation*}
which can be expressed as
\begin{equation}\label{amde1}
f(z) = \frac{2(z+8)}{(4-z)^2} + \frac{24\sqrt{z} \arcsin(\sqrt{z}/2)}{(4-z)^{5/2}}.
\end{equation}
The function $f(z)$ was studied in detail recently by Amdeberhan et al. \cite{amde} and Koshy and Gao \cite{koshygao}.
It also appears in the article by Yin and Qi \cite{yinqi} and is linked to an interesting problem proposed
by Beckwith and Harbor in the AMM \cite{beckwith}. Reciprocals of Catalan numbers possess
the following integral representation derived by Qi and Guo \cite{qiguo}
\begin{equation*}
\frac{1}{C_n} = \frac{(2n+3)(2n+2)(2n+1)}{2^{4n+4}} \int_0^2 z^{2n+1}\sqrt{4-z^2} dz.
\end{equation*}

\medskip

Corresponding to $G(z)$ are the sub-series $G_1(z)$, $G_2(z)$, $G_3(z)$ and $G_4(z)$, namely,
\[
G_1 (z) = \sum_{n = 1}^\infty  {C_{2n - 1} \frac{{z^{2n - 1} }}{{4^{2n - 1} }}}  = \frac{2}{z} - \frac{{(\sqrt {1 + z}  + \sqrt {1 - z} )}}{z},
\]

\[
G_2 (z) = \sum_{n = 0}^\infty  {C_{2n} \frac{{z^{2n} }}{{4^{2n} }}}  = \frac{{\sqrt {1 + z}  - \sqrt {1 - z} }}{z},
\]

\[
G_3(z) = 2\sum_{n = 1}^\infty  {( - 1)^{n - 1} C_{2n - 1} \frac{{z^n }}{{4^{2n} }}}  = (1 + z)^{1/4} \cos \left( {\frac{1}{2}\arctan \sqrt z } \right) - 1,\quad|z|\le1,
\]

\[
G_4 (z) = 2\sum_{n = 0}^\infty  {( - 1)^n C_{2n} \frac{{z^n }}{{4^{2n + 1} }}}  = \frac{{(1 + z)^{1/4} }}{{\sqrt z }}\sin \left( {\frac{1}{2}\arctan \sqrt z } \right),\quad|z|\le1.
\]
Since
\[
\cos\Big(\frac12\arctan(\sqrt{p})\Big) =\sqrt{\frac{\sqrt{1+p}+1}{2\sqrt{1+p}}},\quad\sin\Big(\frac12\arctan(\sqrt{p})\Big) =\sqrt{\frac{\sqrt{1+p}-1}{2\sqrt{1+p}}},
\]
we have the more compact formulas
\[
G_3(z) = 2\sum_{n = 1}^\infty {( - 1)^{n - 1} C_{2n - 1} \frac{{z^n }}{{4^{2n} }}} = \sqrt{\frac{\sqrt{1+z}+1}{2}}-1
\]
and
\[
G_4 (z) = 2\sum_{n = 0}^\infty {( - 1)^n C_{2n} \frac{{z^n }}{{4^{2n + 1} }}} = \sqrt{\frac{\sqrt{1+z}-1}{2z}}.
\]
Our purpose in this paper is to study $G(z)$, $G_1(z)$, $G_2(z)$, $G_3(z)$, $G_4(z)$, $f(z)$ and the following similar series

\[
X(z) = \sum_{n = 1}^\infty {\frac{{z^n }}{{n(n + 1)C_n }}} ,\quad Y(z) = \sum_{n = 1}^\infty {\frac{{z^n }}{{n^2 (n + 1)C_n }}},\quad |z|<4,
\]

\[
W(z) = \sum_{n = 0}^\infty {\frac{{C_n }}{{2^{2n + 1} }}\frac{{z^{2n + 2} }}{{2n + 1}}} ,\quad |z|<1,
\]

focusing mainly on delivering new Fibonacci-Catalan relations. Similar series were studied by Qi and Guo \cite[Section 7]{qiguo} and Stewart \cite[Section 4]{stewart21}.
Other recently published works on infinite sums with (reciprocal) Catalan numbers and central binomial coefficients
include the articles by Zhao and Wang \cite{zhao}, Sprugnoli \cite{sprug}, Chu and Zheng \cite{chu2}, Boyadzhiev \cite{boy1,boy2},
Chen \cite{chen}, Frontczak et al. \cite{fro2} and Uhl \cite{uhl}. Some Fibonacci-Catalan series were also derived, among other results, in a recent article by Adegoke et al.\cite[Section 2]{KA_RF_TG_Catalan2}.

\medskip

We recall that Fibonacci numbers $F_n$ and the companion sequence of Lucas numbers $L_n$ are defined for $n\geq 0$
as $F_{n+2} = F_{n+1} + F_n$ and $L_{n+2} = L_{n+1} + L_n$ with initial conditions $F_0 = 0, F_1 = 1$,
$L_0=2$ and $L_1=1$, respectively. These sequences have ids \seqnum{A000045} and \seqnum{A000032} in the
On-Line Encyclopedia of Integer Sequences \cite{OEIS}. The Binet formulas are given by
\begin{equation*}
F_n = \frac{\alpha^n-\beta^n}{\alpha-\beta}, \qquad L_n = \alpha^n + \beta^n,
\end{equation*}
where $\alpha$ is the golden ratio, i.e., $\alpha = \frac{1+\sqrt{5}}{2}$ and $\beta=-1/\alpha=\frac{1-\sqrt{5}}{2}$.
See \cite{koshy1} for more details.

\medskip

The next lemma will be used frequently.

\begin{lemma}[{\cite[Lemma 1]{adegoke_inv}}, see also {\cite[p.271, identities (20)--(22)]{srivastava12}}]\label{lem.iqkulim}
We have
\begin{gather}
\sin \left( {\frac{\pi }{{10}}} \right) =  - \frac{\beta }{2},\quad\sin \left( {\frac{{3\pi }}{{10}}} \right) = \frac{\alpha }{2} = \alpha ^2 \sin \left( {\frac{\pi }{{10}}} \right),\\
\cos \left( {\frac{\pi }{{10}}} \right) = \frac{{\sqrt {\alpha \sqrt 5 } }}{2},\quad\cos \left( {\frac{{3\pi }}{{10}}} \right) = \frac{{\sqrt { - \beta \sqrt 5 } }}{2} =  - \beta \cos \left( {\frac{\pi }{{10}}} \right),\\
\cot (2\pi /5) =  - \beta ^3 \cot (\pi /5) =  - \beta ^3 \sqrt {\frac{{\alpha ^3 }}{{\sqrt 5 }}}\label{eq.vn4pi2d} .
\end{gather}
\end{lemma}

\section{Results from $G(z)$, $G_1(z)$, $G_2(z)$, $G_3(z)$ and $G_4(z)$}

Observe that $G(1/5)$ and $G(-1/5)$ give
\begin{equation}
\sum_{n = 0}^\infty  {\frac{{C_n }}{{5^n }}}  =  - \beta \sqrt 5
\end{equation}
and
\begin{equation}
\sum_{n = 0}^\infty  {\frac{{( - 1)^n C_n }}{{5^n }}}  = \beta ^2 \sqrt 5 ,
\end{equation}
from which we also infer
\begin{equation}
\sum_{n = 0}^\infty  {\frac{{C_{2n} }}{{5^{2n} }}}  = \frac{{\sqrt 5 }}{2}
\end{equation}
and
\begin{equation}
\sum_{n = 1}^\infty  {\frac{{C_{2n - 1} }}{{5^{2n - 1} }}}  = \frac{{ - \beta ^3 \sqrt 5 }}{2}.
\end{equation}
The trigonometric version of the generating function $G(z)$ of the Catalan numbers is
\[
G_t (z) = \sum_{n = 0}^\infty  {C_n \frac{{\sin ^{2n} z}}{{2^{2n} }}}  =\frac1{\cos ^2 \left( {\frac{z}{2}} \right)},\quad 0 < z < \pi /2.
\]
At $z=\pi/2$, $z=\pi/3$, $z=\pi/4$ and $z=\pi/6$, we have
\begin{equation}
\sum_{n = 0}^\infty  {\frac{{C_n }}{{4^n }}}  = 2,\quad\sum_{n = 0}^\infty  {\frac{{3^n C_n }}{{4^{2n} }}}  = \frac{4}{3},
\end{equation}

\begin{equation}
\sum_{n = 0}^\infty  {\frac{{C_n }}{{8^n }}}  = 4 - 2\sqrt 2 ,\quad\sum_{n = 0}^\infty  {\frac{{C_n }}{{4^{2n} }}}  = 4(2 - \sqrt 3 ).
\end{equation}

The identities $G_1(1/\sqrt5)$ and $G_2(1/\sqrt5)$ give
\begin{equation}
\sum_{n = 1}^\infty  {\frac{{C_{2n - 1} }}{{5^n 4^{2n - 1} }}}  = 2 - \sqrt 2 \sqrt {\frac{{\alpha ^3 }}{{\sqrt 5 }}},\quad\sum_{n = 0}^\infty  {\frac{{C_{2n} }}{{5^n 4^{2n} }}}  = \sqrt 2 \sqrt { - \beta ^3 \sqrt 5 }.
\end{equation}

The trigonometric versions of $G_1(z)$ and $G_2(z)$, for $0\le z\le\pi/2$, are
\[
G_{1t} (z) = \sum_{n = 1}^\infty  {C_{2n - 1} \frac{{\sin ^{2n - 1} z}}{{4^{2n - 1} }}}  = \frac{{4\sin ^2 (z/4)}}{{\sin z}}
\]
and
\[
G_{2t} (z) = \sum_{n = 1}^\infty  {C_{2n} \frac{{\sin ^{2n} z}}{{4^{2n} }}}  = \frac{1}{{\cos (z/2)}}-1.
\]

\begin{example}
Evaluating these versions at appropriate arguments yields for instance
\begin{equation}
\sum_{n = 1}^\infty \frac{C_{2n-1}3^n}{2^{6n-3}}  = (\sqrt{3} - 1)^2,\quad\sum_{n = 1}^\infty \frac{C_{2n}}{2^{4n}}  = \sqrt{2},
\end{equation}
\begin{equation}
\sum_{n = 1}^\infty \frac{C_{2n}3^n}{2^{6n}}  = \frac{2 \sqrt{3}}{3},\quad\sum_{n = 1}^\infty \frac{C_{2n}}{2^{6n}}  = \sqrt{2}(\sqrt{3} - 1).
\end{equation}

\end{example}

\begin{lemma}\label{lem.fdf14ud}
We have
\begin{equation}\label{eq.z60zcvo}
\sin ^2 \left( {\frac{{3\pi }}{{20}}} \right) = \frac{1}{2}\left( {1 - \frac{1}{2}\sqrt { - \beta \sqrt 5 } } \right),
\end{equation}
\begin{equation}\label{eq.tqoyixe}
\sin ^2 \left( {\frac{\pi }{{20}}} \right) = \frac{1}{2}\left( {1 - \frac{1}{2}\sqrt {\alpha \sqrt 5 } } \right),
\end{equation}
\begin{equation}\label{eq.leua6ck}
\sin ^2 \left( {\frac{{3\pi }}{{20}}} \right) - \sin ^2 \left( {\frac{\pi }{{20}}} \right) = \frac{{\sqrt { - \beta^3 \sqrt 5 } }}{4},
\end{equation}
\begin{equation}\label{eq.dr721bq}
\sin ^2 \left( {\frac{{3\pi }}{{20}}} \right) = \left( {1 + \sqrt {\alpha \sqrt 5 } } \right)\sin ^2 \left( {\frac{\pi }{{20}}} \right).
\end{equation}
\end{lemma}
\begin{proof}
Identities \eqref{eq.z60zcvo} and~\eqref{eq.tqoyixe} are straightforward consequences of
\[
\sin^2\left(\frac x2\right)=\frac12(1 - \cos x).
\]
Identity \eqref{eq.leua6ck} comes from
\[
\sin ^2 (3x) - \sin ^2 x = \sin (2x)\sin (4x).
\]
\end{proof}

\begin{theorem}
For integer $s$,
\begin{equation}
\sum_{n = 0}^\infty  {\frac{{C_n F_{2n + s} }}{{4^{2n + 2} }}}  = \frac{{F_{s - 2} }}{2}\left( {1 - \frac{1}{2}\sqrt {\alpha \sqrt 5 } } \right) + \frac{{\alpha ^{s - 2} }}{4}\sqrt {\frac{{ - \beta ^3 }}{{\sqrt 5 }}},
\end{equation}

\begin{equation}
\sum_{n = 0}^\infty  {\frac{{C_n L_{2n + s} }}{{4^{2n + 2} }}}  = \frac{{L_{s - 2} }}{2}\left( {1 - \frac{1}{2}\sqrt {\alpha \sqrt 5 } } \right) + \frac{{\alpha ^{s - 2} }}{4}\sqrt { - \beta ^3 \sqrt 5 }.
\end{equation}
\end{theorem}
\begin{proof}
Determine $\alpha^sG_t(3\pi/10)\pm\beta^sG_t(\pi/10)$, where $s$ is an arbitrary integer, using the Binet formulas and Lemma \ref{lem.fdf14ud}.
\end{proof}

\begin{example}
We have
\begin{equation}
\sum_{n = 0}^\infty  {\frac{{C_n F_{2n} }}{{4^{2n + 2} }}}  = \frac{1}{4}\sqrt {\frac{{ - \beta ^3 }}{{\sqrt 5 }}},
\end{equation}

\begin{equation}
\sum_{n = 0}^\infty  {\frac{{C_n L_{2n} }}{{4^{2n + 2} }}}  = 1 - \frac{1}{2}\sqrt {\alpha \sqrt 5 }  + \frac{1}{4}\sqrt { - \beta ^3 \sqrt 5 },
\end{equation}

\begin{equation}
\sum_{n = 0}^\infty  {\frac{{C_n F_{2n + 3} }}{{4^{2n + 2} }}}  = \frac{1}{2}\left( {1 - \frac{1}{2}\sqrt {\alpha \sqrt 5 } } \right) + \frac{1}{4}\sqrt {\frac{{ - \beta }}{{\sqrt 5 }}},
\end{equation}

\begin{equation}
\sum_{n = 0}^\infty  {\frac{{C_n L_{2n + 3} }}{{4^{2n + 2} }}}  = \frac{1}{2}\left( {1 - \frac{1}{2}\sqrt {\alpha \sqrt 5 } } \right) + \frac{1}{4}\sqrt { - \beta \sqrt 5 }.
\end{equation}
\end{example}
At $z=1$ and $z=1/3$, $G_3(z)$ gives
\begin{equation}
\sum_{n = 1}^\infty  {\frac{{( - 1)^{n - 1} C_{2n - 1} }}{{4^{2n} }}}  = \frac{1}{2}\sqrt {\frac{{1 + \sqrt 2 }}{2}}  - \frac{1}{2},
\end{equation}
and
\begin{equation}
\sum_{n = 1}^\infty  {\frac{{( - 1)^{n - 1} C_{2n - 1} }}{{3^n 4^{2n} }}}  = \frac{{(1 + \sqrt 3 )\sqrt {3\sqrt 3 } }}{{12}} - \frac{1}{2}.
\end{equation}
Similarly, $G_4(1)$ and $G_4(1/3)$ give
\begin{equation}
\sum_{n = 0}^\infty  {\frac{{( - 1)^n C_{2n} }}{{4^{2n} }}}  = \sqrt {2\sqrt 2  - 2},
\end{equation}
and
\begin{equation}
\sum_{n = 0}^\infty  {\frac{{( - 1)^n C_{2n} }}{{3^n 4^{2n} }}}  = 3^{1/4} (\sqrt 3  - 1).
\end{equation}
\begin{lemma}\label{lem.eyjaa37}
We have
\[
\sqrt \alpha   = \alpha \sqrt { - \beta },\quad\sqrt {\alpha \sqrt 5 }  = \alpha \sqrt { - \beta \sqrt 5 }.
\]
\end{lemma}
\begin{lemma}\label{lem.emjkwrk}
For integer $r$,
\begin{gather}
\alpha ^r  + \beta ^{r - 1}  = \alpha F_{r - 2}  + F_{r + 1},\\
\alpha ^r  - \beta ^{r - 1}  = \alpha F_{r + 1}  - F_{r - 2}.
\end{gather}
\end{lemma}

\begin{theorem}
For integer $s$,
\begin{equation}\label{eq.z2edmdb}
\sum_{n = 0}^\infty  {\frac{{C_{2n} F_{2n + s} }}{{4^{3n} }}}  = (\alpha F_{s + 1}  - F_{s - 2} )\sqrt {\frac{2}{5}}\,\sqrt { - \beta \sqrt 5 }  - L_{s - 2}\sqrt {\frac{2}{5}},
\end{equation}

\begin{equation}\label{eq.px5m9w2}
\sum_{n = 0}^\infty  {\frac{{C_{2n} L_{2n + s} }}{{4^{3n} }}}  = (\alpha F_{s - 2}  + F_{s + 1} )\sqrt { - 2 \beta \sqrt 5 }  - F_{s - 2} \sqrt {10},
\end{equation}

\begin{equation}\label{eq.soir62i}
\sum_{n = 1}^\infty  {\frac{{C_{2n - 1} F_{2n + s} }}{{4^{3n - 1} }}}  = 2F_s  - \frac{{L_{s - 1} }}{{\sqrt {10} }} - (\alpha F_{s + 2}  - F_{s - 1} )\sqrt {\frac{{ - \beta }}{{2\sqrt 5 }}},
\end{equation}

\begin{equation}\label{eq.ydexc96}
\sum_{n = 1}^\infty  {\frac{{C_{2n - 1} L_{2n + s} }}{{4^{3n - 1} }}}  = 2L_s  - \frac{1}{{\sqrt 2 }}F_{s - 1} \sqrt 5  - (\alpha F_{s - 1}  + F_{s + 2} )\sqrt {\frac{{ - \beta \sqrt 5 }}{2}}.
\end{equation}
\end{theorem}
\begin{proof}
With $s$ an arbitrary integer and noting Lemma \ref{lem.eyjaa37}, $\alpha ^{s - 1} G_2 (\alpha /2) \mp \beta ^{s - 1} G_2( - \beta /2)$ means
\[
\sum_{n = 0}^\infty  {C_{2n} \frac{{(\alpha ^{2n + s}  \mp \beta ^{2n + s} )}}{{4^{3n} }}}  = \sqrt 2 (\alpha ^s  \mp \beta ^{s - 1} )\sqrt { - \beta \sqrt 5 }  - \sqrt 2 (\alpha ^{s - 2}  \pm \beta ^{s - 2} );
\]
and hence, using Lemma \ref{lem.emjkwrk} and the Binet formulas, identities \eqref{eq.z2edmdb} and \eqref{eq.px5m9w2}. The proof of \eqref{eq.soir62i} and \eqref{eq.ydexc96} is similar. Use $\alpha ^{s} G_1 (\alpha /2) \mp \beta ^{s} G_1( - \beta /2)$.
\end{proof}

\begin{example}
Three examples for $s=0,1,2$ are
\begin{equation}
\sum_{n = 0}^\infty \frac{C_{2n} F_{2n}}{4^{3n}} = \sqrt{\frac{2}{5}} \Big (\sqrt{\alpha^3 \sqrt 5 }  - 3 \Big ),\quad\sum_{n = 0}^\infty \frac{C_{2n} L_{2n}}{4^{3n}} = \sqrt{2} \Big (\sqrt{\frac{\sqrt{5}}{\alpha^3}}  + \sqrt{5} \Big ),
\end{equation}

\begin{equation}
\sum_{n = 0}^\infty \frac{C_{2n} F_{2n+2}}{4^{3n}} = 2 \sqrt{\frac{2}{5}} \Big (\sqrt{\alpha \sqrt 5 }  - 1 \Big ),\quad\sum_{n = 0}^\infty \frac{C_{2n} L_{2n+2}}{4^{3n}} = 2 \sqrt{2} \sqrt{\frac{\sqrt{5}}{\alpha}},
\end{equation}
\begin{equation}
\sum_{n = 1}^\infty \frac{C_{2n-1} F_{2n+1}}{4^{3n-1}} = 2 + \sqrt{\frac{2}{5}} - \sqrt{2} \sqrt{\frac{\alpha}{ \sqrt 5 }},\quad\sum_{n = 1}^\infty \frac{C_{2n-1} L_{2n+1}}{4^{3n-1}} = 2 - \sqrt{2} \sqrt{\frac{\sqrt{5}}{\alpha}}.
\end{equation}

\end{example}

\section{Results from $f(z)$}

It is convenient to write $f(z)$ as
\[
\begin{split}
f(z) = \sum_{n = 0}^\infty  {\frac{{z^n }}{{C_n }}}  &= \frac{{2(z + 8)}}{{z^2 }}\cot ^4 (\arccos (\sqrt z /2))\\
&\qquad + \frac{{24}}{{z^2 }}\cot ^5 (\arccos (\sqrt z /2))\arcsin (\sqrt z /2).
\end{split}
\]
\begin{theorem}
For integer $s$,
\begin{gather}
\sum_{n = 0}^\infty  {\frac{{F_{2n + s} }}{{C_n }}}  = \frac{2}{5}(F_{s + 4}  + 8F_{s + 2} ) + (F_{s + 3} \sqrt 5  - F_{s + 1} )\frac{{12\pi }}{{25}}\sqrt {\frac{{\alpha ^3 }}{{\sqrt 5 }}},\\
\nonumber\\
\sum_{n = 0}^\infty  {\frac{{L_{2n + s} }}{{C_n }}}  = \frac{2}{5}(L_{s + 4}  + 8L_{s + 2} ) + (L_{s + 3} \sqrt 5  - L_{s + 1} )\frac{{12\pi }}{{25}}\sqrt {\frac{{\alpha ^3 }}{{\sqrt 5 }}}.
\end{gather}
\end{theorem}
\begin{proof}
Considering $\alpha^sf_t(\alpha^2)$ and $\beta^sf_t((-\beta)^2)$, where $s$ is an arbitrary integer, we have
\begin{gather}
\sum_{n = 0}^\infty  {\frac{{\alpha ^{2n + s} }}{{C_n }}}  = \frac{{2(\alpha ^{s + 2}  + 8\alpha ^s )}}{{\alpha ^4 }}\cot ^4 \left( {\frac{\pi }{5}} \right) + \frac{{36\pi }}{5}\alpha ^{s - 4} \cot ^5 \left( {\frac{\pi }{5}} \right),\\
\sum_{n = 0}^\infty  {\frac{{\beta ^{2n + s} }}{{C_n }}}  = \frac{{2(\beta ^{s + 2}  + 8\beta ^s )}}{{\beta ^4 }}\cot ^4 \left( {\frac{{2\pi }}{5}} \right) + \frac{{12\pi }}{5}\beta ^{s - 4} \cot ^5 \left( {\frac{{2\pi }}{5}} \right),
\end{gather}
from which, using the Binet formulas and relevant identities from Lemma \ref{lem.iqkulim}, we get
\[
\begin{split}
\sqrt 5 \sum_{n = 0}^\infty  {\frac{{F_{2n + s} }}{{C_n }}}  &= \frac{2}{5}(\alpha ^{s + 4}  - \beta ^{s + 4} ) + \frac{{16}}{5}(\alpha ^{s + 2}  - \beta ^{s + 2} )\\
&\qquad+ (3\alpha ^{s + 2}  + \beta ^{s + 2} )\frac{{12\pi }}{{25}}\sqrt {\frac{{\alpha ^3 }}{{\sqrt 5 }}},
\end{split}
\]
\[
\begin{split}
\sum_{n = 0}^\infty  {\frac{{L_{2n + s} }}{{C_n }}}  &= \frac{2}{5}(\alpha ^{s + 4}  + \beta ^{s + 4} ) + \frac{{16}}{5}(\alpha ^{s + 2}  + \beta ^{s + 2} )\\
&\qquad+ (3\alpha ^{s + 2}  - \beta ^{s + 2} )\frac{{12\pi }}{{25}}\sqrt {\frac{{\alpha ^3 }}{{\sqrt 5 }}}.
\end{split}
\]
The stated identities in the theorem now follow when we use the Binet formulas and invoke Lemma \ref{lem.oakhmtx} with $r=s+2$.
\end{proof}
\begin{example}
We have
\begin{equation}\label{eq.bbu2w6b}
\sum_{n = 0}^\infty  {\frac{{F_{2n} }}{{C_n }}}  = \frac{{22}}{5} + \frac{{12\pi (4\alpha  - 3)}}{{25}}\sqrt {\frac{{\alpha ^3 }}{{\sqrt 5 }}},\quad\sum_{n = 0}^\infty  {\frac{{L_{2n} }}{{C_n }}}  = \frac{{62}}{5} + \frac{{12\pi (8\alpha  - 5)}}{{25}}\sqrt {\frac{{\alpha ^3 }}{{\sqrt 5 }}},
\end{equation}
\begin{equation}
\sum_{n = 0}^\infty  {\frac{{F_{2n - 1} }}{{C_n }}}  = 4 + \frac{{12\pi }}{{25}}\sqrt {\alpha ^3 \sqrt 5 },\quad\sum_{n = 0}^\infty  {\frac{{L_{2n - 1} }}{{C_n }}}  = \frac{{24}}{5} + \frac{{12\pi (6\alpha  - 5)}}{{25}}\sqrt {\frac{{\alpha ^3 }}{{\sqrt 5 }}},
\end{equation}
\begin{equation}
\sum_{n = 0}^\infty  {\frac{{F_{2n - 2} }}{{C_n }}}  = \frac{2}{5} + \frac{{24\pi }}{{25}}\sqrt {\frac{\alpha }{{\sqrt 5 }}},\quad\sum_{n = 0}^\infty  {\frac{{L_{2n - 2} }}{{C_n }}}  = \frac{{38}}{5} + \frac{{24\pi }}{{25}}\sqrt {\frac{{\alpha ^5 }}{{\sqrt 5 }}}.
\end{equation}
\end{example}

The first identity in~\eqref{eq.bbu2w6b} was also obtained by Stewart~\cite[Identity~(38a)]{stewart21}.

Since $L_{2n}=5F_n^2+2(-1)^n$, $L_{2n}=L_n^2-2(-1)^n$ and
$$
\sum_{n=0}^\infty \frac{(-1)^n}{C_n}=\frac{14}{25}-\frac{24\sqrt5}{125}\ln (\alpha)
$$
we have also the following interesting series
\begin{equation}
5\sum_{n=1}^\infty \frac{F^2_{n}}{C_n}=\frac{282}{25}+\frac{6(15+19\sqrt5)\pi}{125} \,\omega + \frac{48\sqrt5}{125}\ln (\alpha),
\end{equation}
\begin{equation}
\sum_{n=0}^\infty \frac{L^2_{n}}{C_n}=\frac{338}{25}+\frac{6(15+19\sqrt5)\pi}{125} \,\omega - \frac{48\sqrt5}{125}\ln (\alpha),
\end{equation}
where $\omega=\sqrt{\sqrt5\alpha}=\sqrt{2+\alpha}$.

\section{Results from $W(z)$}

\[
W(z) = \sum_{n = 0}^\infty  {\frac{{C_n }}{{2^{2n + 1} }}\frac{{z^{2n + 2} }}{{2n + 1}}}  = z\arcsin z + \sqrt {1 - z^2 }  - 1.
\]
The trigonometric version of $W(z)$ is
\[
W_t (z) = \sum_{n = 0}^\infty  {\frac{{C_n }}{{2^{2n + 1} }}\frac{{\sin ^{2n + 2} z}}{{2n + 1}}}  = z\sin z - 2\sin ^2 \left( {\frac{z}{2}} \right).
\]
At $z=\pi/2$, $\pi/3$, $\pi/4$, $\pi/6$, $W_t$ gives
\begin{equation}
\sum_{n = 0}^\infty  {\frac{{C_n }}{{2^{2n + 1} (2n + 1)}}}  = \frac{\pi }{2} - 1,
\end{equation}

\begin{equation}
\sum_{n = 0}^\infty  {\frac{{C_n 3^n }}{{2^{4n} (2n + 1)}}}  = \frac{{4(\pi \sqrt 3  - 3)}}{9},
\end{equation}

\begin{equation}
\sum_{n = 0}^\infty  {\frac{{C_n }}{{2^{3n} (2n + 1)}}}  = \frac{{(\pi  + 4)\sqrt 2 }}{2} - 4,
\end{equation}

\begin{equation}
\sum_{n = 0}^\infty  {\frac{{C_n }}{{2^{4n} (2n + 1)}}}  = \frac{{2(\pi  + 6\sqrt 3  - 12)}}{3}.
\end{equation}
\begin{lemma}\label{lem.equ4w3l}
For integer $r$,
\begin{gather}
3\alpha^r+\beta^r=2L_r+F_r\sqrt5,\\
3\alpha^r-\beta^r=L_r+2F_r\sqrt5.
\end{gather}

\end{lemma}
\begin{theorem}
For integer $s$,
\begin{equation}
\begin{split}
\sum_{n = 0}^\infty  {\frac{{C_n }}{{4^{2n + 1} }}\frac{{F_{2n + s} }}{{2n + 1}}}  &= \frac{\pi }{{10\sqrt 5 }}(2L_{s - 1}  + F_{s - 1} \sqrt 5 )\\
&\qquad + F_{s - 2} (\sqrt {\alpha \sqrt 5 }  - 2)- \alpha ^{s - 2} \sqrt {\frac{{ - \beta ^3 }}{{\sqrt 5 }}},
\end{split}
\end{equation}
\begin{equation}
\begin{split}
\sum_{n = 0}^\infty  {\frac{{C_n }}{{4^{2n + 1} }}\frac{{L_{2n + s} }}{{2n + 1}}}  &= \frac{\pi }{{10}}(L_{s - 1}  + 2F_{s - 1} \sqrt 5 )\\
&\qquad + L_{s - 2} (\sqrt {\alpha \sqrt 5 }  - 2)- \alpha ^{s - 2} \sqrt { - \beta ^3 \sqrt 5 }.
\end{split}
\end{equation}
\end{theorem}
\begin{proof}
Determining $\alpha ^{s - 2} 2W_t (3\pi /10) \mp \beta ^{s - 2} 2W_t (\pi /10)$, where $s$ is an arbitrary integer, and employing identity \eqref{eq.leua6ck}, we have
\[
\begin{split}
\sum_{n = 0}^\infty  {\frac{{C_n }}{{4^{2n + 1} }}\frac{{\alpha ^{2n + s}  \mp \beta ^{2n + s} }}{{2n + 1}}}  &= \frac{\pi }{{10}}(3\alpha ^{s - 1}  \pm \beta ^{s - 1} )\\
&\qquad - 4(\alpha ^{s - 2}  \mp \beta ^{s - 2} )\sin ^2 \left( {\frac{\pi }{{20}}} \right) - \sqrt { - \beta ^3 \sqrt 5 } \alpha ^{s - 2};
\end{split}
\]
from which the stated identities follow in view of the Binet formulas and Lemmas~\ref{lem.fdf14ud} and~\ref{lem.equ4w3l}.
\end{proof}
\begin{example}
We have
\[
\sum_{n = 1}^\infty  {\frac{{C_{n - 1} }}{{4^{2n - 1} }}\frac{{F_{2n} }}{{2n - 1}}}  = \frac{{\pi \alpha ^3 }}{{10\sqrt 5 }} - \sqrt {\frac{{ - \beta ^3 }}{{\sqrt 5 }}}.
\]

\end{example}

\section{Results from $Y(z)$}
The identity
\[
Y(z) = \sum_{n = 1}^\infty  {\frac{{z^n }}{{n^2 (n + 1)C_n }}}  = 2(\arcsin (\sqrt z /2))^2
\]
immediately yields the following summation formulas:
\begin{equation}\label{zhao}
\sum_{n = 1}^\infty {\frac{1}{{n^2 (n + 1)C_n }}} = \frac{\pi^2}{18}, \qquad
\sum_{n = 1}^\infty {\frac{(-1)^n}{{n^2 (n + 1)C_n }}} = -2 \ln^2\Big (\frac{1}{\alpha} \Big ),
\end{equation}
\begin{equation}
\sum_{n = 1}^\infty \frac{{4^n }}{{n^2 (n + 1)C_n }} = \frac{{\pi ^2 }}{2}, \qquad
\sum_{n = 1}^\infty \frac{{3^n }}{{n^2 (n + 1)C_n }} = \frac{{2\pi ^2 }}{9}.
\end{equation}
Most likely all these summations are known. The first one is a classical result due to Euler.
Both sums in equation \eqref{zhao} appear in the article by Zhao and Wang \cite{zhao}.

\begin{theorem}
For integer $s$,
\begin{equation}
\sum_{n = 1}^\infty {\frac{{F_{2n + s} }}{{n^2 (n + 1)C_n }}} = \frac{{\pi ^2 }}{{50\sqrt 5 }}(9\alpha ^s  - \beta ^s ),
\end{equation}

\begin{equation}
\sum_{n = 1}^\infty  {\frac{{L_{2n + s} }}{{n^2 (n + 1)C_n }}} = \frac{{\pi ^2 }}{{50}}(9\alpha ^s  + \beta ^s ).
\end{equation}
\end{theorem}
\begin{proof}
Determine $\alpha^sY(\alpha^2)\pm\beta^sY(\beta^2)$ and use the Binet formulas.
\end{proof}
\begin{example}
We have
\begin{equation}\label{eq.a5dtfdc}
\sum_{n = 1}^\infty  {\frac{{F_{2n} }}{{n^2 (n + 1)C_n }}}  = \frac{{4\pi ^2 }}{{25\sqrt 5 }},
\end{equation}

\begin{equation}
\sum_{n = 1}^\infty  {\frac{{L_{2n} }}{{n^2 (n + 1)C_n }}}  = \frac{{\pi ^2 }}{5}.
\end{equation}
\end{example}
Identity~\eqref{eq.a5dtfdc} was also obtained by Stewart~\cite[Identity~(37c)]{stewart21}.
\section{Results from $X(z)$}
\[
X(z) = \sum_{n = 1}^\infty  {\frac{{z^n }}{{n(n + 1)C_n }}}  = \frac{{2\sqrt z \arcsin (\sqrt z /2)}}{{\sqrt {4 - z} }},
\]
which can also be written
\[
X(z) = \sum_{n = 1}^\infty  {\frac{{z^n }}{{n(n + 1)C_n }}}  = 2\cot (\arccos (\sqrt z/2))\arcsin (\sqrt z/2).
\]
\begin{lemma}\label{lem.oakhmtx}
For any integer $r$,
\begin{gather}
3\alpha ^r  - \beta ^{r + 3}  = L_{r + 1} \sqrt 5  - L_{r - 1},\\
3\alpha ^r  + \beta ^{r + 3}  = \sqrt 5 (F_{r + 1} \sqrt 5  - F_{r - 1} ).\label{eq.rch7jzs}
\end{gather}
\end{lemma}
\begin{theorem}
For integer $s$,
\begin{equation}
\sum_{n = 1}^\infty  {\frac{{F_{2n + s} }}{{n(n + 1)C_n }}}  = (F_{s + 1} \sqrt 5  - F_{s - 1} )\frac{\pi }{5}\sqrt {\frac{{\alpha ^3 }}{{\sqrt 5 }}} ,
\end{equation}

\begin{equation}
\sum_{n = 1}^\infty  {\frac{{L_{2n + s} }}{{n(n + 1)C_n }}}  = (L_{s + 1} \sqrt 5  - L_{s - 1} )\frac{\pi }{5}\sqrt {\frac{{\alpha ^3 }}{{\sqrt 5 }}}.
\end{equation}
\end{theorem}
\begin{proof}
$\alpha^sX(\alpha^2)\pm\beta^sX(\beta^2)$ gives
\[
\sum_{n = 1}^\infty  {\frac{{\alpha ^{2n + s}  \pm \beta ^{2n + s} }}{{n(n + 1)C_n }}}  = \frac{{3\pi }}{5}\alpha ^s \cot \left( {\frac{\pi }{5}} \right) \pm \frac{\pi }{5}\beta ^s \cot \left( {\frac{{2\pi }}{5}} \right).
\]
Thus, using the Binet formulas and identity \eqref{eq.vn4pi2d} of Lemma \ref{lem.iqkulim}, we have
\begin{equation}
\sqrt 5 \sum_{n = 1}^\infty  {\frac{{F_{2n + s} }}{{n(n + 1)C_n }}}  = \frac{\pi }{5}(3\alpha ^s  + \beta ^{s + 3} )\cot \left( {\frac{\pi }{5}} \right),
\end{equation}

\begin{equation}
\sum_{n = 1}^\infty  {\frac{{L_{2n + s} }}{{n(n + 1)C_n }}}  = \frac{\pi }{5}(3\alpha ^s  - \beta ^{s + 3} )\cot \left( {\frac{\pi }{5}} \right);
\end{equation}
and hence the stated identities, upon use of Lemma \ref{lem.oakhmtx} with $r=s$.
\end{proof}
\begin{example}
We have
\begin{equation}\label{Cat_Fib}
\sum_{n = 1}^\infty  {\frac{{F_{2n} }}{{n(n + 1)C_n }}}  = \frac{{2\pi }}{5}\sqrt {\frac{\alpha }{{\sqrt 5 }}},\quad\sum_{n = 1}^\infty  {\frac{{L_{2n} }}{{n(n + 1)C_n }}}  = \frac{{2\pi }}{5}\sqrt {\frac{{\alpha ^5 }}{{\sqrt 5 }}},
\end{equation}
\begin{equation}
\sum_{n = 1}^\infty  {\frac{{F_{2n - 1} }}{{n(n + 1)C_n }}}  = \frac{\pi }{5}\sqrt {\frac{{\alpha ^3 }}{{\sqrt 5 }}},\quad\sum_{n = 1}^\infty  {\frac{{L_{2n - 1} }}{{n(n + 1)C_n }}}  =  - (4\beta  + 1)\frac{\pi }{5}\sqrt {\frac{{\alpha ^3 }}{{\sqrt 5 }}}.
\end{equation}
\end{example}
Both series in \eqref{Cat_Fib} appeared recently as a problem proposal \cite{fro1}. 

The first identity in \eqref{Cat_Fib} was also found by Stewart~\cite[Identity~(37b)]{stewart21}.

\section{Some combinatorial identities}

Before closing we state some combinatorial identities (finite and infinite)
which can be inferred from the series studied in the previous sections.
Concerning the finite class we note that similar results were studied by Witu\l a and S\l ota
and their collaborators \cite{witula1,witula2} and more recently by Alzer and Nagy \cite{alzer},
Chu \cite{chu1}, Batir et al. \cite{batir1} and Batir and Sofo \cite{batir2}.
Our first example is an identity derived by Witu\l a and S\l ota \cite{witula1} using a completely different method.

\begin{theorem}\label{mainCI1}
For each $n\geq 1$,
\begin{equation}\label{C_id1}
\sum_{k=1}^n \frac{2^{2k}}{\binom {2k}{k}} = \frac{1}{3}\Big ( \frac{2^{2n+1}}{C_n} - 2\Big ).
\end{equation}
\end{theorem}
\begin{proof}
We work with the function $f(z)$. From \eqref{amde1} we get
\begin{equation}
\sum_{n=0}^\infty \frac{z^{2n}}{C_n} = \frac{2(z^2+8)}{(4-z^2)^2} + \frac{24z \arcsin(z/2)}{(4-z^2)^{5/2}}.
\end{equation}
Now, for all $|z|<1$,
\begin{eqnarray*}
2z (1-z^2)^{-5/2} \arcsin(z) & = & (1-z^2)^{-2} \frac{2z}{\sqrt{1-z^2}}\arcsin(z) \\
& = & \Big (\sum_{n=0}^\infty (n+1) z^2\Big ) \Big (\sum_{n=1}^\infty \frac{(2z)^{2n}}{n\binom {2n}{n}} \Big ) \\
& = & \sum_{n=1}^\infty \sum_{k=1}^n (n+1-k) \frac{2^{2k}}{k\binom {2k}{k}} z^{2n}
\end{eqnarray*}
and by replacing $z$ by $z/2$
\begin{equation*}
16 z (4-z^2)^{-5/2} \arcsin(z/2) = \frac{1}{2} \sum_{n=1}^\infty \sum_{k=1}^n (n+1-k) \frac{2^{2k-2n}}{k\binom {2k}{k}} z^{2n}.
\end{equation*}
Hence,
\begin{equation*}
\sum_{n=0}^\infty \frac{z^{2n}}{C_n} = \frac{2(z^2+8)}{(4-z^2)^2} + \frac{3}{4} \sum_{n=1}^\infty \sum_{k=1}^n
(n+1-k) \frac{2^{-2(n-k)}}{k\binom {2k}{k}} z^{2n}.
\end{equation*}
Next, from the partial fraction decomposition
\begin{equation*}
\frac{16+2z^2}{(4-z^2)^2} = \frac{1}{4(2+z)} + \frac{1}{4(2-z)} + \frac{3}{2(2+z)^2} + \frac{3}{2(2-z)^2}
\end{equation*}
it follows that
\begin{equation*}
\frac{16+2z^2}{(4-z^2)^2} = \sum_{n=0}^\infty (2+3n) 2^{-(2n+1)} z^{2n}, \qquad |z|<2.
\end{equation*}
Comparing the coefficients of $z^{2n}$ and rearranging yields for all $n\geq 1$
\begin{equation*}
\sum_{k=1}^n (n+1-k) \frac{2^{2k}}{k\binom {2k}{k}} = \frac{1}{3}\Big ( \frac{2^{2n+2}}{C_n} - 4 \Big ) - 2n.
\end{equation*}
The identity \eqref{C_id1} follows from
\begin{equation*}
\sum_{k=1}^n \frac{2^{2k}}{2k\binom {2k}{k}} = \frac{2^{2n}}{\binom {2n}{n}} - 1,
\end{equation*}
which is known as Parker's formula (see \cite{witula3}).
\end{proof}

\begin{theorem}\label{mainCI2}
For each $n\geq 0$,
\begin{equation}\label{C_id2}
\sum_{k=0}^n \frac{\binom {2k}{k} \binom {2(n-k)}{n-k}}{(2k+1)(2(n-k)+1)}  = \frac{16^{n}}{(n+1)(2n+1) \binom {2n}{n}}.
\end{equation}
\end{theorem}
\begin{proof}
This identity can be derived straightforwardly working with the function $Y(z)$ in conjunction with the arcsin series
\begin{equation*}
A(z)=\arcsin (z) = \sum_{n=0}^\infty \binom {2n}{n} \frac{z^{2n+1}}{4^n (2n+1)}, \qquad |z|<1.
\end{equation*}
\end{proof}

It is interesting to compare identity \eqref{C_id2} with the following identity derived by Witu\l a et al. \cite{witula2}
and rediscovered by Batir et al. \cite{batir1} applying the Wilf-Zeilberger method:
\begin{equation*}
\sum_{k=0}^n \frac{\binom {2k}{k} \binom {2(n-k)}{n-k}}{(2k+1)}  = \frac{16^{n}}{(2n+1) \binom {2n}{n}}.
\end{equation*}
The results stated in Theorems \ref{thm.o98dcfn} -- \ref{thm.b0yjj17} follow from identities $G_t(z)$, $G_{1t}(z)$, $G_{2t}(z)$ and Y(z), in view of the identity
\begin{equation}\label{eq.d1gh5ds}
\int_0^{\pi /2} {\sin ^{2n + 1} x\,dx}  = \frac{1}{{2n + 1}}\frac{{2^{2n} }}{\binom{2n}n},\quad n\ge 0.
\end{equation}

\begin{theorem}\label{thm.o98dcfn}
If $r$ is a positive integer, then,
\[
\sum_{n = 0}^\infty  {\frac{2}{{(n + 1)(2n + 2r + 1)}}\frac{\binom{2n}n}{\binom{2n + 2r}{n + r}}}  = \frac{1}{{(2r - 1)\binom{2r - 2}{r - 1}}} - \frac{1}{{2^{2r - 1} r}}.
\]
\end{theorem}

\begin{theorem}
If $r$ is a non-negative integer, then,
\[
\sum_{n = 1}^\infty  {\frac{1}{{(2n + 1)(2n + 2r + 1)}}\frac{\binom{4n}{2n}}{\binom{2n + 2r}{n + r}}} \frac{1}{{2^{2n} }} = \frac{1}{{2^{2r - 1} }}\int_0^{\pi /2} {\sin ^{2r} x\sin (x/2)\,dx}  - \frac{1}{{(2r + 1)\binom{2r}{r}}}.
\]

\end{theorem}
In particular,
\begin{equation}
\sum_{n = 1}^\infty  {\frac{1}{{(2n + 1)^2}}\frac{\binom{4n}{2n}}{\binom{2n}{n}}} \frac{1}{{2^{2n} }} =3-2\sqrt 2.
\end{equation}
\begin{theorem}
If $r$ is a positive integer, then,
\[
\sum_{n = 1}^\infty  {\frac{1}{{n(2n + 2r - 1)}}\frac{\binom{4n - 2}{2n - 1}}{\binom{2n + 2r - 2}{n + r - 1}}} \frac{1}{{2^{2n} }} = \frac{1}{{(2r - 1)\binom{2r - 2}{r - 1}}} - \frac{1}{{2^{2r - 2} }}\int_0^{\pi /2} {\sin ^{2r - 1} x\cos (x/2)\,dx}.
\]
\end{theorem}
In particular,
\begin{equation}
\sum_{n = 1}^\infty  {\frac{1}{{n(2n + 1)}}\frac{\binom{4n - 2}{2n - 1}}{\binom{2n}{n}}} \frac{1}{{2^{2n} }} =\frac{\sqrt 2 - 1}3.
\end{equation}

\begin{theorem}\label{thm.b0yjj17}
If $r$ is an integer and $r\ge -1$, then,
\[
\sum_{n = 1}^\infty  {\frac{1}{{n^2 (2n + 2r + 1)}}} \frac{{2^{2n} }}{\binom{2n}n}\frac{{2^{2n} }}{\binom{2n + 2r}{n + r}} = \frac{1}{{2^{2r - 1} }}\int_0^{\pi /2} {x^2 \sin ^{2r + 1} x\,dx}.
\]

\end{theorem}
In particular,
\begin{equation}
\sum_{n = 1}^\infty  {\left( {\frac{{2^{2n} }}{{n\binom{2n}n}}} \right)^2 \frac{1}{{(2n + 1)}}}  = 2\pi  - 4,
\end{equation}

\begin{equation}
\sum_{n = 1}^\infty  {\left( {\frac{{2^{2n} }}{n}} \right)^2 \frac{1}{{\binom{2n}n\binom{2n - 2}{n - 1}}}\frac{1}{{2n - 1}}}  =  - \frac{7}{2}\zeta (3) + 2\pi G,
\end{equation}
where
\[
G = \sum_{i = 0}^\infty  {\frac{{( - 1)^i }}{{(2i + 1)^2 }}}
\]
is Catalan's constant.

The trigonometric form of $A(z)$ is
\begin{equation}\label{eq.sff956y}
A_t (z) = \sum_{n = 0}^\infty  {\binom{2n}n\frac{{\sin ^{2n + 1} z}}{{4^n (2n + 1)}}}  = z.
\end{equation}
\begin{lemma}[Lewin {\cite[Identity A.3.3.13]{lewin81}}]\label{lem.ofcbhbu}
For real or complex $y$,
\[
\int_0^y {\frac{x}{{\sin x}}dx}  = Cl_2 (y) + Cl_2 (\pi  - y) + y\log \left( {\tan \left( {\frac{y}{2}} \right)} \right),
\]
where $Cl_2(z)$ is the Clausen function defined by
\[
Cl_2 (z) =  - \int_0^z {\log |2\sin (x/2)|\,dx}.
\]
\end{lemma}
The following values are known, among others, see Lewin~\cite[p.291, Section A.2.4]{lewin81}:
\[
Cl_2 (\pi/2)=G=-Cl_2 (3\pi/2).
\]
Our next result, Theorem \ref{thm.vm3zcb8}, is a straightforward consequence of $A_t(z)$, given in~\eqref{eq.sff956y}, upon application of Lemma~\ref{lem.ofcbhbu} and the well known result:
\begin{equation}\label{eq.amu7tsd}
\int_0^{\pi /2} {\sin ^{2n} x\,dx}  = \frac{\pi }{{2^{2n + 1} }}{\binom{2n} n}\quad n = 0,1,2, \ldots
\end{equation}
\begin{theorem}\label{thm.vm3zcb8}
If $r$ is an integer and $r\ge 0$, then,
\[
\sum_{n = 0}^\infty  {\dfrac{\binom{2n}n\binom{2n + 2r}{n + r}}{{2^{4n + 2r + 1} (2n + 1)}}}  = \frac1{\pi }\int_0^{\pi /2} {x\sin ^{2r - 1} x\,dx}.
\]
\end{theorem}
In particular,
\begin{equation}
\sum_{n = 0}^\infty  {\frac{{\binom{2n}n}^2}{{2^{4n + 1} (2n + 1)}}}  = \frac{2G}\pi,
\end{equation}

\begin{equation}
\sum_{n = 0}^\infty  {\frac{{\binom{2n}n\binom{2n + 2}{n + 1}}}{{2^{4n + 3} (2n + 1)}}}  = \frac{1}{\pi }.
\end{equation}
From the function $W_t(z)$ and identities~\eqref{eq.d1gh5ds} and~\eqref{eq.amu7tsd} comes the next result.
\begin{theorem}
If $r$ is an integer and $r\ge -1$, then,
\[
\sum_{n = 0}^\infty  {\dfrac{{\binom{2n}n\binom{2n + 2r + 2}{n + r + 1}}}{{(n + 1)(2n + 1)2^{4n + 2r + 4} }}}  = \frac{1}{\pi }\int_0^{\pi /2} {z\sin ^{2r + 1} zdz}  - \frac{1}{{2^{2r + 1} }}\binom{2r}r + \frac{1}{{\pi (2r + 1)}}.
\]

\end{theorem}
In particular,
\begin{equation}
\sum_{n = 0}^\infty  {\frac{{\binom{2n}n}^2}{{2^{4n + 2}(n + 1)(2n + 1)}}}  = \frac{2G - 1}\pi,
\end{equation}

\begin{equation}
\sum_{n = 0}^\infty  {\frac{{\binom{2n}n\binom{2n + 2}{n + 1}}}{{2^{4n + 4}(n + 1)(2n + 1)}}}  = \frac{2}{\pi } - \frac12.
\end{equation}


\begin{thebibliography}{99}

\bibitem{adegoke_inv}
K.~Adegoke, Fibonacci identities involving reciprocals of binomial coefficients, preprint, arXiv:2112.00622, 2021.

\bibitem{KA_RF_TG_Catalan2}
K. Adegoke, R. Frontczak and T. Goy, On a family of infinite series with reciprocal Catalan numbers, \emph{Axioms} {\bf 11}, 165, (2022), \verb+https://doi.org/10.3390/axioms11040165+.

\bibitem{alzer}
H. Alzer and G. V. Nagy, Some identities involving central binomial coefficients and Catalan numbers, \emph{Integers}, {\bf 20}
(2020), \#A59, 17 pages.

\bibitem{amde}
T. Amdeberhan, X. Guan, L. Jiu, V. H. Moll and C. Vignat, A series involving Catalan numbers:
Proofs and demonstrations, \emph{Elem. Math.} {\bf 71} (2016), 109--121.

\bibitem{batir1}
N. Batir, H. K\"uc\"uk and S. Sorgun, Convolution identities involving the central binomial coefficients and Catalan numbers,
\emph{Trans. Combinatorics} {\bf 10} (2021), 225--238.

\bibitem{batir2}
N. Batir and A. Sofo, Finite sums involving reciprocals of the binomial and central binomial
coefficients and harmonic numbers, \emph{Symmetry} {\bf 13 2002} (2021), 13 pages.

\bibitem{beckwith}
D. Beckwith and S. Harbor, Problem 11765, \emph{Amer. Math. Monthly} {\bf 121 (3)} (2014), 267--267.

\bibitem{boy1}
K. N. Boyadzhiev, Series with central binomial coefficients, Catalan numbers, and harmonic numbers,
\emph{J. Integer Seq.} {\bf 15} (2012), Article 12.1.7.

\bibitem{boy2}
K. N. Boyadzhiev, Power series with inverse binomial coefficients and harmonic numbers,
\emph{Tatra Mountains Math. Publ.} {\bf 70 (1)} (2017), 199--206.

\bibitem{chen}
H. Chen, Interesting series associated with central binomial coefficients, Catalan numbers and harmonic numbers,
\emph{J. Integer Seq.} {\bf 19} (2016), Article 16.1.5.

\bibitem{chu1}
W. Chu, Alternating convolutions of Catalan numbers, \emph{Bull. Braz. Math. Soc. (N.S.)}, 2021.

\bibitem{chu2}
W. Chu and D. Zheng, Infinite series with harmonic numbers and central binomial coefficients, \emph{Int. J. Number Theory}
{\bf 5} (2009), 429--448.

\bibitem{dana1}
T. Dana-Picard, Parametric integrals and Catalan numbers, \emph{Int. J. Math. Educ. Sci. Technol.} {\bf 36} (2005), 410--414.

\bibitem{dana2}
T. Dana-Picard, Integral presentations of Catalan numbers, \emph{Int. J. Math. Educ. Sci. Technol.} {\bf 41} (2010), 63--69.

\bibitem{dana3}
T. Dana-Picard, Integral presentations of Catalan numbers and Wallis formula, \emph{Int. J. Math. Educ. Sci. Technol.}
{\bf 42} (2011), 122--129.

\bibitem{fro1}
R. Frontczak, Problem H-874, \emph{Fibonacci Quart.} {\bf 59} (2021), 185--185.

\bibitem{fro2}
R. Frontczak, H. M. Srivastava and Z. Tomovski, Some families of Ap\'{e}ry-like Fibonacci and Lucas series,
\emph{Mathematics} {\bf 9 1621} (2021), 10 pages.

\bibitem{koshy1}
T. Koshy, \emph{Fibonacci and Lucas Numbers with Applications}, Wiley-Interscience, 2001.

\bibitem{koshy2}
T. Koshy, \emph{Catalan Numbers with Applications}, Oxford University Press, Oxford, 2009.

\bibitem{koshygao}
T. Koshy and Z. G. Gao, Convergence of a Catalan series, \emph{College Math. J.,} {\bf 43 (2)} (2012), 141--146.

\bibitem{lewin81} L.~Lewin, \emph{Polylogarithms and associated functions}, Elsevier Inc., (1981).

\bibitem{qiguo}
F. Qi and B.-N. Guo, Integral representations of the Catalan numbers and their applications, \emph{Mathematics,} {\bf 40 (5)} (2017), 31 pages.

\bibitem{stanley}
R. P. Stanley, \emph{Catalan Numbers}, Cambridge University Press, Cambridge, UK, 2015.

\bibitem{sprug}
R. Sprugnoli, Sums of reciprocals of the central binomial coefficients, \emph{Integers}, {\bf 6} (2006), \#A27, 18 pages.

\bibitem{srivastava12}
H. M. Srivastava and J. Choi (2012). \emph{Zeta and q-Zeta Functions and Associated Series and Integrals}, Elsevier Inc.

\bibitem{stewart21}
S. M. Stewart, The inverse versine function and sums containing reciprocal central binomial coefficients
and reciprocal Catalan numbers, \emph{Int. J. Math. Educ. Sci. Technol.}, (2021), \verb+https://doi.org/10.1080/0020739X.2021.1912842+.

\bibitem{uhl}
M. Uhl, Recurrence equation and integral representation of Ap\'{e}ry sums,
\emph{European J. Math.} {\bf 7} (2021), 793--806.

\bibitem{witula1}
R. Witu\l a and D. S\l ota, Finite sums connected with the inverses of central binomial numbers and Catalan numbers,
\emph{Asian-Europ. J. Math.} {\bf 1} (2008), 439--448.

\bibitem{witula2}
R. Witu\l a, E. Hetmaniok, D. S\l ota and N. Gawro\'{n}ska, \emph{Convolution identities for central binomial numbers,} Int. J.
Pure Appl. Math. {\bf 85} (2013), 171--178.

\bibitem{witula3}
R. Witu\l a, D. S\l ota, J. Matlak, A. Chmielowska and M. R\'{o}\.{z}a\'{n}ski, Matrix methods in evaluation of integrals,
\emph{J. Appl. Math. Comp. Mech.}, {\bf 19 (1)} (2020), 103--112.

\bibitem{yinqi}
L. Yin and F. Qi, Several series identities involving the Catalan numbers, \emph{Trans. A. Razmadze Math. Inst.} {\bf 172} (2018), 466--474.

\bibitem{zhao}
F.-Z. Zhao and T. Wang, Some results for sums of the inverses of binomial coefficients, \emph{Integers}, {\bf 5} (2005), \#A22, 5 pages.

\bibitem{OEIS}
N. J. A. Sloane, \textit{The On-Line Encyclopedia of Integer Sequences}, \url{https://oeis.org}.

\end{thebibliography}
\end{document}